\documentclass[12pt]{amsart}

\usepackage{amsmath,amsfonts,amssymb,amsthm}
\usepackage{pinlabel}
\usepackage{graphicx}
\usepackage[usenames,dvipsnames]{color}
\usepackage[all]{xy}
\usepackage{hyperref}
\usepackage{color}
\xyoption{dvips}

\numberwithin{equation}{section}

\newtheorem{thm}{Theorem}[section]
\newtheorem{cor}[thm]{Corollary}
\newtheorem{prop}[thm]{Proposition}
\newtheorem{lem}[thm]{Lemma}

\theoremstyle{remark}

\theoremstyle{definition}

\theoremstyle{remark}
\newtheorem{remark}[thm]{Remark}
\newtheorem{conj}[thm]{Conjecture}

\newcommand{\R}{\mathbb{R}}

\newcommand{\Z}{\mathbb{Z}}

\newcommand{\N}{\mathbb{N}}
\newcommand{\T}{\mathbb{T}}

\newcommand{\id}{\mathrm{id}}

\newcommand{\OP}{\operatorname}

\begin{document}

\title[Hamiltonian fixed points on symplectically aspherical manifolds]{The number of Hamiltonian fixed points on symplectically aspherical manifolds}

\author{Georgios Dimitroglou Rizell}
\author{Roman Golovko}

\begin{abstract}
We show that a generic Hamiltonian diffeomorphism on a closed symplectic manifold which is symplectically aspherical has at least the stable Morse number of fixed points -- this is in line with a conjecture by Arnold. 
\end{abstract}

\address{Uppsala University, Sweden}
\email{georgios.dimitroglou@math.uu.se}
\urladdr{http://www.dimitroglou.name/}
\address{Universit\'{e} libre de Bruxelles, Belgium} \email{rgolovko@ulb.ac.be}
\urladdr{https://sites.google.com/site/ragolovko/}

\date{\today}


\maketitle

\section{Introduction}
In the 1960's, Arnold announced several fruitful conjectures in symplectic topology concerning the number of fixed point of a Hamiltonian diffeomorphism in both the absolute case (concerning periodic Hamiltonian orbits) and the relative case (concerning Hamiltonian chords on a Lagrangian submanifold). These questions originate in, among others, questions in celestial mechanics. The strongest form of Arnold's conjecture for a closed symplectic manifold can be stated as follows:
\begin{conj}[Appendix 9 in \cite{MMOCM} \& p. 284 in \cite{AP}]
\label{conj}
The number of fixed points of an arbitrary (resp. generic) Hamiltonian
diffeomorphism of a closed symplectic manifold $(X,\omega)$ is
greater or equal than the number of critical points of a smooth
(resp. Morse) function $X \to \R$.
\end{conj}
The minimal number of critical points of a Morse function (resp. stable Morse function) on a closed smooth manifold $Χ$ is referred to as the \emph{Morse number} (resp. \emph{stable Morse number}), and we denote it by $\OP{Morse}(X)$ (resp. $\OP{stabMorse}(X)$); see Section \ref{morsenumber} for more details. Similarly, the minimal number of critical points of an arbitrary smooth function is denoted by $\OP{Crit}(X)$.

Many strong results have been established in line with the above
conjecture; a brief recollection is given in Section
\ref{sec:previous}. The goal of this paper is to provide a
sharpening of the previously known results in certain special cases
of particularly well-behaved symplectic manifolds.

Most of the previously known results have been established by using Floer homology. This theory was originally invented by Floer \cite{MTFFPOSD} in order to study precisely this problem, and it is based upon Gromov's theory of pseudoholomorphic curves \cite{PCISM}.

The results in this paper are also proven using a version of Floer homology, namely a version for pairs of Lagrangian submanifolds with coefficients twisted in the fundamental group. The details of our setup are given in Section \ref{bifurcationanalysissection}. The Lagrangian intersection Floer homology was invented by Floer in \cite{MTFLI} to study the relative version of the problem. We also refer to \cite{FOOO1}, \cite{FOOO2} by Fukaya--Oh--Ohta--Ono for the construction of Floer homology in a much more general setting than the one considered here.

In \cite{DRGSM}, the authors obtained the following result in the relative setting, i.e.~concerning
the intersection points of Lagrangian submanifolds, under the strong assumption that the Lagrangian submanifolds
are exact: The number of intersection points of an exact Lagrangian submanifold
$L \subset (X,\omega=d\eta)$ and $\phi(L)$ for a generic Hamiltonian
diffeomorphism $\phi$ is bounded from below by the \emph{stable}
Morse number of $L,$ under the assumptions that $L$ is relatively
spin and has vanishing Maslov class.

In this paper, we prove the following result: 
\begin{thm}
\label{thm:main} Let $L\subset (X,\omega)$ be a relative spin
Lagrangian submanifold of a symplectic manifold with bounded
geometry at infinity, which satisfies the condition that both the
Maslov class and the symplectic area class vanish on $\pi_2(X,L)$.
Then
$$ L \pitchfork \phi^1_{H_t}(L) \ge \OP{stabMorse(L)}$$
for any generic Hamiltonian diffeomorphism $\phi^1_{H_t}$.
\end{thm}
\begin{remark}
\label{rem} In the case of a simply connected manifold $M$ of
dimension either $\dim M=1,2$ or $\dim M \ge 6,$ the stable Morse
number is equal to the Morse number. Namely, in these cases the
result of Smale \cite{OTSOM} shows that, moreover, the Morse number
is equal to the minimal number of generators of a complex over $\Z$
whose homology is isomorphic to $H_\bullet(M).$ In the non-simply
connected case, there are counter examples due to Damian
\cite{OTSMNOACM}.
\end{remark}

In the case when $X=T^*M$ is a cotangent bundle equipped with the canonical symplectic form and when $L$ is the zero section, this was proven by Laudenbach--Sikorav in \cite{PERSISTANCE}. When $L \subset (X,d\lambda)$ is exact (and hence $X$ is necessarily open), together with the assumptions of the above theorem, the result was proven by the authors in \cite[Theorem 1.3]{DRGSM}. In order to see the difference between Theorem~\ref{thm:main} and the result proven in \cite{DRGSM}, observe that $L \subset (X,\omega=d\eta)$ being exact is a  very restrictive condition since, in particular, $X$ then necessarily is an open manifold.

Observe that, in the case when $L$ in addition is simply connected and of dimension at least $6$, the result can be seen to follow from Floer's original lower bound on the number of intersection points in terms of the generators of the homology of $L$ in combination with Smale's result in Remark \ref{rem} above.

In this paper our main goal concerns the original conjecture due to Arnold concerning fixed points of Hamiltonian
diffeomorphisms. In certain settings, it will turn out that we can
use the above result in order to obtain a strong partial answer.

We say that a symplectic manifold $(X,\omega)$ is
\emph{symplectically aspherical} if both the first Chern class and
the symplectic area class vanish on the elements of $\pi_2(X)$. A
symplectic manifold $(X,\omega)$ is said to be \emph{symplectically
atoroidal} under the assumption that these classes vanish when
pulled back to any map $\T^2 \to X$ from a torus. In particular,
since there is a map $\T^2 \to S^2$ of degree one, every
symplectically atoroidal manifold is symplectically aspherical.

Using Floer's approach to study Hamiltonian diffeomorphisms on a
closed symplectic manifold $(X,\omega)$ by studying Lagrangian
intersections of the diagonal $\Delta \subset (X \times
X,\omega\oplus(-\omega))$, we get the following result for
symplectically aspherical closed symplectic manifolds.
\begin{cor}
\label{cor:main}
Let $(X,\omega)$ be a closed symplectic manifold which is symplectically aspherical. Then, any generic Hamiltonian diffeomorphism $\phi^1_{H_t}$ of $X$ has at least $\OP{stabMorse(X)}$ number of fixed points.
\end{cor}

Our result concerning the Lagrangian intersections is based on the same ingredients as those used to prove the results in \cite{DRGSM} by the authors. More precisely, we use that the invariance proof of Floer homology via bifurcation analysis established by Sullivan in \cite{KTIFFH} shows that the Floer homology complex is simple homotopy equivalent to the Morse complex. On the other hand, first note that bifurcation analysis in \cite{KTIFFH} is formulated in the setting
where the Hamiltonian perturbation term vanishes (i.e.~with a
differential that counts pseudoholomorphic strips), and hence we must use the dictionary between the Hamiltonian chord formulation of Floer
homology and the version formulated in terms of intersection points.
In addition, we must adapt the construction of Sullivan and perform bifurcation analysis for a Floer subcomplex consisting of only the contractible Hamiltonian chords; for more details see Section~\ref{bifurcationanalysissection}. The result \cite[Theorem 2.2]{OTSMNOACM} of Damian, based upon classical differential topology, can then be used in order to create a stable Morse function whose corresponding Morse homology complex is isomorphic to the Floer complex.

This result has been independently proven by  P\"{o}der \cite{P}. His proof is based on a different version of bifurcation analysis that is performed directly in the non-relative setting.

Finally, we note that bifurcation analysis in the setting of Floer homology also has been carried out by Lee in \cite{YJL1}, \cite{YJL2}, who used this to study torsion. 

\section{Background and previous results}
\label{sec:previous}
\subsection{Morse number and stable Morse number}
\label{morsenumber} First, we recall the definition of the stable
Morse number:
\begin{align*}
\OP{stabMorse}(X):=\min\limits_{k \ge 0 \atop F \in C^\infty(X
\times \R^k,\R)}\{\# \OP{Crit}(F)\ | \: F \: \mbox{is Morse and
almost quadratic at infinity}\}.
\end{align*}
A Morse function $F \colon X \times \R^k \to \R$ is said to be almost quadratic at infinity if there is a uniform bound on the norm $\|dF-dQ\|_{C^0}$ for some fixed non-degenerate (possibly indefinite) quadratic form $Q$ on $\R^k$. It follows immediately from the definition that
\begin{align}\label{morsevsstabmorse}\OP{Morse}(X)\geq \OP{stabMorse}(X).
\end{align}
It is known that Morse number and stable Morse number coincide in a
number of cases, i.e. for surfaces, for simply connected closed
three-manifolds (i.e.~the three-sphere) by Perelman's solution of
the Poincar\'{e} conjecture \cite{TEFFTRFAIGA}, and for simply
connected closed $k$-manifolds, where $k\geq 6$, see \cite{OTSOM}.
(Note that a hypothetical exotic smooth 4-sphere must have a Morse
number strictly smaller than the stable Morse number.) On the
contrary, in \cite[Theorem 1.2]{OTSMNOACM} Damian proved that closed
manifolds with a `complicated' fundamental group has the property
that $\OP{Morse}(X)> \OP{stabMorse}(X)$. The result of Damian
combined with the result of Gompf \cite{ANCOSM} stating that every
finitely presentable group can be realised as a fundamental group of
a closed symplectic $2n$-manifold, $n\geq 2$, leads to the examples
of closed symplectic $2n$-manifolds, $n\geq 2$, for which
Inequality~\ref{morsevsstabmorse} is not sharp. In addition, note
that from Morse theory it follows that both $\OP{Morse}(X)$ and
$\OP{stabMorse}(X)$ are bounded from below by the sum of Betti
numbers of $X$. For a more detailed discussion on $\OP{Morse}(X)$
and $\OP{stabMorse}(X)$ we refer the reader to \cite{OTSMNOACM,
LISFDA}.

\subsection{Generalities from symplectic geometry}
A \emph{symplectic manifold} consists of a pair $(X,\omega)$ of a $2n$-dimensional smooth manifold $X$, together with a closed two-form $\omega$ which is non-degenerate in the sense that $\omega^{\wedge{n}}$ is a volume form. A \emph{Lagrangian submanifold} $L \subset (X,\omega)$ is an $n$-dimensional smooth submanifold onto which $\omega$ pulls back to zero. For open symplectic manifolds, the notion of \emph{bounded geometry at infinity} was introduced in \cite{PCISM} in order to guarantee that the pseudoholomorphic curves behave well.

Let $(X,\omega)$ be a symplectic manifold. A
symplectomorphism $(X, \omega) \to (X, \omega)$ is called
Hamiltonian if it is a time-1 map $\phi^{1}_{H_t}$ of the flow of the
Hamiltonian vector field $X_{H_t}$ defined by $i_{X_{H_t}}
\omega=-dH_t$, where $H:\mathbb R\times X \to \mathbb R$ is a smooth
function and $H_t(x):=H(t,x)$. We define $\OP{Fix}(\phi^1_{H_t})$ to
be the number of fixed points of $\phi^1_{H_t}$ in $X$.

\subsection{Previous results in the closed case}
First, note that Arnold observed that the conjecture is true when $H$ is a $C^2$-small function. In this case, the conjecture follows from elementary differential topology.

The attempts to solve the strong Arnold conjecture gave rise to new ideas and techniques in the calculus of variations and in nonlinear elliptic systems. In particular, this conjecture led Floer to the construction of his homology theory \cite{MTFFPOSD, ARMIFTSA, TUGFOTSA, MTFLI}.

Floer's  original result \cite[Theorem 1]{SFPAHS} establishes Corollary \ref{cor:main} under the additional assumption that $X$ is simply connected and of dimension at least $\dim X \ge 6$. Namely, in this case the Morse number is determined by the integer graded homology groups with coefficients in $\Z$ (see Remark \ref{rem}).

By the work of Rudyak--Oprea \cite{OTLSCOSMATAC}, the version of the
Arnold conjecture in the non-generic case is true for
symplectically aspherical closed symplectic manifolds, i.e.~the
bound
$$\OP{Fix}(\phi^1_{H_t}) \geq \OP{Crit}(X)$$
is satisfied for a general Hamiltonian diffeomorphism.

Besides the strong Arnold conjecture, there are many other estimates for a generic Hamiltonian diffeomorphism of a similar flavor which in general are weaker than $\OP{Morse}(X)$. We only mention a few, and refer the reader to \cite{MTFHS}, \cite{SIAHD}, and \cite{LOFH} for more details.

A very general lower bound that has been established is
\begin{align*}
\OP{Fix}(\phi^1_{H_t})\geq \sum\limits_{i} \dim H_i(X,\mathbb Q),
\end{align*}
which was proven for a general closed symplectic manifold in Fukaya--Ono \cite{ACAGWI, ACAGWIFGSM}, Liu-Tian \cite{FHAAC}, and Ruan \cite{VNAPHC}, building on the work by Floer. Also, see \cite{SFDTAQC} by Piunikhin--Salamon--Schwarz.

In more recent results the above general bound has been improved by
taking also the fundamental group of the symplectic manifold into
account (e.g.~providing bounds related to the minimal number of generators);
see the works of Ono--Pajitnov \cite{OTFPOAHDIPOFG} and
Barraud \cite{AFFG}.

\section{Examples of symplectically aspherical manifolds}
Here we mention  a few examples of symplectically aspherical (and
even symplectically atoroidal) manifolds to which our main result
can be applied.
\subsubsection*{Surfaces and fibre bundles}
Observe that any symplectic surface $\Sigma_g$ of genus $g\geq 1$
has the property that $\pi_2(\Sigma_g)=0$, and is hence
symplectically aspherical. Since products of such surfaces also have
vanishing $\pi_2$, they provide examples in arbitrary even
dimensions.  Even more generally, any symplectic fibre bundle
$$\Sigma_{g_1}\to E \to \Sigma_{g_2}$$ with a fibre $\Sigma_{g_1}$ and
a base $\Sigma_{g_2}$, where $g_1,g_2\geq 1$, the long exact
sequence of a fibration implies that $\pi_2(E)=0$. Note that in
certain situations $E$ is symplectic (and hence symplectically
aspherical) because of the following result due to
Thurston~\cite{SSEOSM}:
\begin{thm}[\cite{SSEOSM}]
A smooth $\Sigma_{g_1}$-bundle $\pi: E\to \Sigma_{g_2}$ admits a
$\pi$-compatible (i.e. compatible with the bundle structure)
symplectic form if and only if the image $[\Sigma_{g_1}] \in
H_2(E;\R)$ of the fibre is non-zero in homology.
\end{thm}

\subsubsection*{Homogeneous spaces}
The next type of examples are homogeneous spaces. First, recall that
a closed (not necessarily symplectic) $2n$-manifold $X$ with $a\in
H^2(M,\R)$ such that $a^n\neq 0$ is called $c$-symplectic. Consider
a Lie group $G$ and a uniform lattice $\pi$ on G (i.e. $\pi$ is a
discreet subgroup of $G$ such that $G/\pi$ is compact). In addition,
assume that $G$ is completely solvable (i.e. any adjoint linear
operator of the Lie algebra $\mathfrak g$ has only real
eigenvalues).

The following statement has been proven by
Ib\'{a}\~{n}ez--K\c{e}dra--Rudyak--Tralle:
\begin{thm}{\cite[Lemma 4.2]{OFGOSAM}}
For a completely solvable simply connected Lie group $G$ and a
uniform lattice $\pi<G$ such that $X:=G/\pi$ is c-symplectic, $X$
admits a symplectic structure and $X=K(\pi_1(X),1)$. In particular,
$\pi_2(X)=0$, and therefore $X$ is symplectically aspherical.
\end{thm}

\subsubsection*{Symplectic manifolds of negative sectional curvature}
The next series of examples comes from the following lemma.
\begin{lem}\label{nonexoftoriusgugi}
If $\pi_2(X)=0$ and $\pi_1(X)$ contains no subgroup isomorphic to $\Z^2$, then any element in $H^2(X,\Z)$ and $H^2(X,\R)$ pull back to zero under a map $f \colon \T^2 \to X$ from a torus. In particular, this is the case when $\pi_1(X)$ is hyperbolic and the universal cover of $X$ is contractible.
\end{lem}
\begin{proof}
Since $f_* (\pi_1(\T^2))$ is a subgroup isomorphic to $\Z_k \oplus \Z_l$ for some $k \in \N \cup \{\infty\}$ and $l \in \N$, it follows that $l\cdot f$ can be represented by a spherical class. Using the assumption $\pi_2(X)=0$, we see that
$0=l\cdot f \in H_2(X)$
or, in other words, $[f]$ is torsion.
\end{proof}
The archetypal example symplectic manifolds satisfying the
assumptions of Lemma~\ref{nonexoftoriusgugi}, and which hence are
symplectically atoroidal, are manifolds of negative sectional
curvature (e.g.~K\"{a}hler manifolds with such curvature
properties).

\subsubsection*{Manifolds with non-vanishing $\pi_2$} Besides the already mentioned examples for which $\pi_2=0$, Gompf constructed examples
of closed symplectic manifolds which are symplectically aspherical
but that have nontrivial $\pi_2$, see Example 6 in \cite{OSAMWNP2}.

\section{Bifurcation analysis for the Floer complex of contractible Hamiltonian chords and proof of Theorem~\ref{thm:main}}
\label{bifurcationanalysissection}
The additional difficulty when dealing with the setting of Theorem~\ref{thm:main} proven here compared to that of \cite[Theorem 1.3]{DRGSM} is the following: it is
necessary to exclude certain generators in order to have a
well-defined graded complex with coefficients in the group ring
$\Z[\pi_1(L)]$ of the fundamental group. More precisely, we need to
consider only those generators of the Floer homology complex
corresponding to the \emph{contractible} Hamiltonian chords. Since
the invariance of the simple homotopy type follows from Sullivan's
bifurcation analysis \cite{KTIFFH}, it will moreover be necessary to
use the dictionary from the Hamiltonian chord formulation of Floer
homology to the version formulated in terms of intersection points.
Even though these techniques are not new, we still outline them in order to make our setup precise.

In addition, we refer to the work \cite{ELCAPI} by Su\'{a}rez and
later \cite{SHEONL} by Abouzaid--Kragh, where the simple homotopy
type of exact Lagrangian submanifolds are studied using Lagrangian
Floer homology.

\subsection{The Floer complex with twisted coefficients}
Lagrangian Floer homology with coefficients twisted by the fundamental group were first studied in \cite{KTIFFH} by Sullivan, and later in \cite{FHOTUC} by Damian and \cite{NLWVMCAHQ} by Abouzaid.

In the following, we consider two closed and relative spin Lagrangian submanifolds $L_i \subset (X,\omega),$ $i=0,1,$ inside a tame symplectic manifold, together with a Hamiltonian isotopy $\phi^t_{H_s}$ of $(X,\omega).$ Both of $L_i,$ $i=0,1,$ are moreover assumed to satisfy the property that the Maslov class as well as the symplectic area class vanish on $\pi_2(X,L_i).$

For a generic Hamiltonian $H_t,$ it is the case that $L_0 \cap \phi^t_{H_t}(L_1)$ consists of transverse intersections consisting of double-points for all $t \in [0,1]$, except for a finite number of non-transverse birth/death moments; we refer to \cite{KTIFFH} for more details.

We now consider a generic two-parameter family $J_{s,t},$ $(s,t) \in [0,1]^2,$ of tame almost complex structures on $(X,\omega),$ together with the one-parameter family
$$G_{s,t}:=sH_{st} \colon X \to \R$$
of time-dependent Hamiltonians induced by $H_t.$ For a generic fixed number $s_0 \in [0,1]$, we  define the Floer homology complex over the group ring $\Z[\pi_1(L_0)]$ freely generated by the time-one Hamiltonian chords from $L_0$ to $L_1$ of the form
$$ [0,1] \ni t \mapsto p(t)= \phi^t_{G_{s_0,t}}(p), \:\: p \in L_0, \: \phi^1_{G_{s_0,t}}(p) \in L_1,$$
living in a fixed component $\mathfrak{p}$ of the space of paths $\Pi(X,L_0,L_1)$ from $L_0$ to $L_1.$ We denote this complex by
$$CF^{\mathfrak{p}}_\bullet(L_0,L_1;G_{s_0,t},J_{s_0,t}):=\Z[\pi_1(L_0)] \langle p(t); \: p(t) \in \mathfrak{p} \in \pi_0(\Pi(X,L_0,L_1)) \rangle. $$

The coefficient $\langle d(p_-),g\cdot p_+ \rangle \in \Z$ in the definition of the differential, where $g \in \pi_1(L_0),$ is given by the count of rigid \emph{Floer strips from $p_-$ to $p_+$ with boundary on $(L_0,L_1)$ defined for the pair $(G_{s_0,t},J_{s_0,t})$} with a boundary in a homotopy class corresponding to $g$ (after closing it up). More precisely, we are interested in the rigid solutions to the boundary value problem
$$
\begin{cases}
u \colon \R \times [0,1] \to X,\\
u(s,i) \in L_i,\:\: i \in \{0,1\},\\
\lim\limits_{s \to \pm \infty} u(s,t)=p_\pm(t),\\
\partial_s u +J_{s_0,t}(\partial_t u(s,t)-X_{G_{s_0,t}}\circ u(s,t))=0,
\end{cases}
$$
for which $u(s,0)$ moreover lives in the class $g \in \pi_1(L_0).$ In order to identify the boundary of the strip with a based loop, we must first close it up by taking concatenations with so-called capping paths. These are auxiliary choices of paths in $L_0$ connecting the end-points of each Hamiltonian chord with the base point.

In the case when the Hamiltonian term $G_{s_0,t} \equiv 0$ vanishes, the above PDE reduces to a Cauchy-Riemann equation for a time-dependent almost complex structure, while the Hamiltonian chords simply are intersection points $L_0 \cap L_1.$ The corresponding solutions will in this case be called \emph{$J_{s_0,t}$-holomorphic strips from $p_-$ to $p_+$ with boundary on $(L_0,L_1)$.}

In the following we will consider the case when $L_0=L_1=L$ and $\mathfrak{p}=0 \in \pi_1(X,L)$ is the class of \emph{contractible} Hamiltonian chords. The assumption that both the symplectic area class and the Maslov class vanish on $\pi_2(X,L)$ makes the Floer theory particularly well-behaved in this setting; namely, there are notions of both a canonical action and a canonical $\Z$-grading.

\emph{A canonical action:} For every Hamiltonian chord $p(t) \in \mathfrak{p}=0$ we fix a \emph{capping disc,} by which we mean a smooth map $c_p \colon D^2 \to X$ for which $e^{it\pi} \in \partial D^2 \cap \{y \ge 0\}$ is mapped to $p(t),$ while $c_p( \partial D^2 \cap \{ y \le 0\}) \in L.$ (Such a disc exists by the contractibility of the chord.) Each generator $p(t)$ can now be equipped with the action
$$ \mathfrak{a}(p):=\int_{c_p} \omega -\int_0^1G_{s_0,t}(p(t))dt $$
which, by the condition that the symplectic area class vanishes on
$\pi_2(X,L),$ is independent of the choice of capping disc. The
\emph{Floer energy} of a strip is the non-negative quantity
$$ E(u):=\int_{-\infty}^{+\infty}\int_0^1 \omega(\partial_s u(s,t),J_{s_0,t} \partial_s(u,t)) \, dt\, ds \ge 0.$$
It is a standard fact that:
\begin{lem}
The Floer energy of Floer strip $u$ from $p_-$ to $p_+$ is
non-negative and equals
$$ E(u)= \mathfrak{a}(p_+)-\mathfrak{a}(p_-) \ge 0.$$
Moreover, its energy vanishes if and only if $p_-=p_+$, and the
strip is the `constant' path $u(s,t)=p_-(t).$ In particular, the
coefficient $\langle d(p_-),g\cdot p_+ \rangle$ vanishes whenever
$\mathfrak{a}(p_-) \ge \mathfrak{a}(p_+)$ is satisfied.
\end{lem}
In other words, with our conventions the differential increases action.

\emph{A canonical $\Z$-grading:} The $\Z$-grading in this setting was originally defined by Floer, see Proposition 2.4 in \cite{MTFLI}. Also see \cite{GLS} by Seidel for later developments. Using the terminology established in the latter, the canonical grading comes from the fact that, in the case when $L_0=L_1$, we are able to use the \emph{same} Maslov potential on both Lagrangians. (The Maslov potential in itself is, however, not canonical.)

\subsection{Naturality}
The following technique provides a dictionary between the Hamiltonian-chord formulation and intersection-point formulation of Lagrangian Floer homology. This technique has been well-known since the early days of Floer theory, and sometimes goes under the name of the `naturality'. This technique is needed since the invariance proof based upon bifurcation analysis is formulated in the setting of intersection points and pseudoholomorphic strips. On the other hand, we need to `remember' the Hamiltonian in order to have a well-defined notion of a contractible generator (for which the action and grading are well-defined).

For a time-dependent almost complex structure $J_t$ and a Hamiltonian isotopy $\phi^t_{H_t}$ we define the induced time-dependent almost complex structure
$$ \widetilde{J}_t := D(\phi^1_H \circ (\phi^t_H)^{-1}) \circ J_t \circ D(\phi^1_H \circ (\phi^t_H)^{-1})^{-1},$$
which is tame if and only if the former is.
\begin{prop}
\label{prp:naturality}
We have the following two bijective correspondences:
\begin{itemize}
\item Hamiltonian time-one chords $\phi^1_{H_t}(p) \in L_1,$ $p \in L_0,$ from $L_0$ to $L_1$ correspond to intersection points $\phi^1_{H_t}(L_0) \cap L_1;$
\item Floer strips between two Hamiltonian chords with boundary on $(L_0,L_1)$ defined by $(H_t,J_t)$, and $\widetilde{J}_t$-holomorphic strips between the corresponding double points with boundary on $(\phi^1_{H_t}(L_0),L_1)$.
\end{itemize}
\end{prop}
\begin{proof}
The pseudoholomorphic strip corresponding to the Floer strip $u(s,t)$ can be expressed as
$$\widetilde{u}(s,t):=(\phi^1_H \circ (\phi^t_H)^{-1})\circ u(s,t).$$
It is readily checked that this is a $\widetilde{J}_t$-holomorphic strip of the form claimed. To that end, one uses the equality
$$\partial_t \widetilde{u}(s,t)=D(\phi^1_H \circ (\phi^t_H)^{-1})\partial_t u(s,t)+\partial_t(\phi^1_H \circ (\phi^t_H)^{-1})_{u(s,t)}$$
combined with
$$ D(\phi^1_H \circ (\phi^t_H)^{-1})^{-1}\partial_t(\phi^1_H \circ (\phi^t_H)^{-1})_{u(s,t)}=-X_{H_t}(u(s,t)),$$
where the latter follows from partially differentiating
$$ (\phi^1_H \circ (\phi^t_H)^{-1})^{-1} \circ (\phi^1_H \circ (\phi^t_H)^{-1})= \id_X $$
with respect to the $t$-variable.
\end{proof}

\subsection{Invariance of the simple homotopy type}
Given a smooth manifold $N$ and a Morse function $f:N\to \R$, let $(CM_\bullet(N,f;\Z[\pi_1(N)]),\partial_f)$ be the Morse homology complex with coefficients in $\Z[\pi_1(N)]$ twisted by the fundamental group of $N$. In particular, this complex is freely generated as a $\Z[\pi_1(N)]$-module by the critical points of $f$. (In general, additional assumptions on $f$ and the metric are necessary in order to have a well-defined differential.)

The following is the invariance result for the above Floer homology complex.
\begin{thm}
Under the assumptions of Theorem \ref{thm:main}, the Floer complex $CF_\bullet^{\mathfrak{p}}(L,L;H_t,J_t)$ for $\mathfrak{p}=0 \in \pi_0(\Pi(X,L,L))$ is well-defined, canonically graded in $\Z$, and simple homotopy equivalent to the Morse complex $(CM_\bullet(L,f;\Z[\pi_1(L)]),\partial_f)$ of $L$.
\end{thm}
\begin{proof}
Similarly to the proof of \cite[Theorem 2.8]{DRGSM}, the main technical ingredient is the invariance in terms of bifurcation analysis as established by Sullivan in \cite{KTIFFH}. Since this formulation concerns a setup of the Floer complex where the Hamiltonian perturbation term vanishes, we must first translate this setup to that considered here. Then we adapted the technique to the Floer complex consisting of only the contractible chords. 

More precisely, we consider a generic `sequence' of complexes
$$CF_\bullet^{\mathfrak{p}}(L,L;G_{s,t},J_{s,t}), \:\: s \in [0,1],$$
$$G_{0,t} \equiv 0, \:\:G_{1,t}=H_t,$$
as above with $\mathfrak{p}=0$. For small $s>0,$ Floer's original calculation in \cite{MTFLI} provide the isomorphism with a Morse homology complex of $L$.

The statement for $s=1$ follows from the invariance in terms of
bifurcation analysis. Roughly speaking, this amounts to explaining
every change in the above sequences of complexes as a finite number
of births/deaths of intersection points together with handle-slides.
A handle-slide consists of a pseudoholomorphic strip connecting two
generators of index $-1$, which generically arise in a
one-parameter family of problems. On the algebraic side, a
birth/death corresponds to a stabilisation with an acyclic complex
consisting of two generators, while a handle-slide corresponds to a
simple isomorphism of complexes. Bifurcation analysis was first made
rigorous in \cite{KTIFFH}, and is there formulated in the setting
where the Hamiltonian perturbation term vanishes (i.e.~with a
differential that counts pseudoholomorphic strips). We must
therefore use Proposition \ref{prp:naturality} in order to translate
between the two viewpoints.

The same proposition also shows that this bifurcation analysis
indeed can be applied to the Floer complexes of contractible chords
only, as considered here, in order to provide well-defined chain
isomorphisms (after stabilisations). In particular, we note that
every pseudoholomorphic handle-slide strip with one of its
asymptotics being a contractible chord must have \emph{both} of its
asymptotics equal to contractible chords. The chain-isomorphisms
constructed by the bifurcation analysis are thus well-defined when restricted to the contractible chords.
\end{proof}
Then we apply the following result proven in
\cite{OTSMNOACM} by Damian.
\begin{thm}[\cite{OTSMNOACM}]\label{resdamadded}
Let $M$ be a closed smooth manifold. Any complex
$(C_\bullet,\partial_C)$ in the same simple homotopy equivalence
class as $(CM_\bullet(M,f;\Z[\pi_1(M)]),\partial_f)$ can be realised
as
\[(C_\bullet,\partial_C) = (CM_\bullet(M \times \R^k,F;\Z[\pi_1(M \times \R^k)]),\partial_F),\]
where $F \colon M \times \R^k \to \R$ is a Morse function almost quadratic at infinity for some $k \ge 0$.
\end{thm}
Note that Damian only considers the case of complexes supported in
degrees $0 \le i \le \dim M$. However, after choosing $k \gg 0$
above sufficiently large, the construction also works for a simple
homotopy equivalence supported in degrees $-N \le i \le N$ for some
arbitrary but fixed $N \ge 0$.

This finishes the proof of Theorem~\ref{thm:main}.

\subsection{Proof of Corollary~\ref{cor:main}}

Any Hamiltonian isotopy $\phi^t_{H_t}$ of $X$ induces a Hamiltonian isotopy
$$\Phi^t:=(\id_X,\phi^t_{H_t}) \colon (X \times X,\omega\oplus(-\omega)) \to (X \times X,\omega\oplus(-\omega))$$
generated by $-H_t \circ \OP{pr}_2,$ where $\OP{pr}_2 \colon X \times X \to X$ denotes the canonical projection to the second factor.

We can now apply Theorem~\ref{thm:main} to the image of the diagonal
$$\Delta \subset (X \times X,\omega\oplus(-\omega)),$$
which is a Lagrangian submanifold that is relatively spin. Note that there is a one-to-one correspondence between generic
fixed points of $\phi^1_{H_t}$ on $X,$ and transverse  intersection points of $\Delta$ and $\Phi^1(\Delta)$.

In addition, observe that the required properties of the Maslov
class and symplectic area class on $\pi_2(X \times X,\Delta)$ follow
by the symplectic asphericity of $(X,\omega)$. More precisely, first
we use the doubling trick. Any disc $(u,v):(D^2, \partial D^2)\to
(X\times X, \Delta)$ gives rise to a sphere  $w: S^2 \to X$ in the
following way. We present $w$ as a pair of discs $u,v:D^2\to X$
glued along the boundary by the identity map. Observe that
$$\langle\omega\oplus (-\omega), (u,v)\rangle=\langle\omega, w\rangle=0,$$
and hence $\langle\omega\oplus (-\omega), \pi_2(X\times X,
\Delta)\rangle=0$. Finally, the required property of the Maslov
class follows in a similar manner; see e.g.~\cite[Lemma 2.3]{FHOLICI}.

\section*{Acknowledgements}
We would like to thank Viktor Ginzburg, Basak Gurel, Kaoru Ono and
Thomas Kragh for the very helpful discussions and interest in our
work. In addition, the second author is grateful to Tohoku Forum for
Creativity, where the project has started, for the hospitality and
support in May 2016. The first author is supported by the grant KAW
2013.0321 from the Knut and Alice Wallenberg Foundation. The second
author is partially supported by the ERC Advanced Grant ``LDTBud''
and by the ERC Consolidator Grant 646649 ``SymplecticEinstein''.

\end{document}